\documentclass[11pt, twoside]{article}
\usepackage{latexsym}
\usepackage{amsmath}
\usepackage{amssymb}
\usepackage[all]{xy}
\usepackage{amsfonts}
\usepackage{verbatim}
\usepackage{amsthm}
\usepackage{mathrsfs}
\usepackage{epsfig}
\usepackage{xy}
\usepackage{tensor}
\usepackage{array}
\usepackage{stmaryrd}
\usepackage{graphicx,color}
\usepackage{xcolor}
\usepackage[colorlinks=true,linkcolor=blue,citecolor=blue]{hyperref}
\usepackage{tikz}
\usetikzlibrary{arrows,calc}
\usepackage{etex}
\usepackage{mathdots}
\usepackage{float}
\usepackage{graphics}
\usepackage{pdflscape}
\usepackage{extarrows}
\usepackage{anysize,hyperref}
\input xypic
\xyoption{all}
\usepackage{bm}
\usepackage[perpage,symbol]{footmisc}
\usepackage{setspace}
\def\T{\mathcal{T}}

\def\P{\mathcal{P}}
\def\I{\mathcal{I}}

\def\C{\mathscr{C}}
\def\E{\mathbb{E}}

\def\s{\mathfrak{s}}

\def\op{^\mathrm{op}}
\def\Ab{\mathsf{Ab}}
\def\del{\delta}
\def\dr{\ar@{->}[r]}

\def\X{\mathscr{X}}
\def\H{\mathscr{H}}

\def\add{\mathsf{add}\hspace{.01in}}

\newcommand{\CC}{{\bf{C}}^{n+2}_{\C}}
\newcommand{\mr}{\hbox{\boldmath$\cdot$}}
\newcommand{\ov}{\overset}
\newcommand{\lra}{\longrightarrow}
\newcommand{\co}{\colon}
\newcommand{\uas}{^{\ast}}            %%% ^*
\newcommand{\sas}{_{\ast}}
\newcommand{\Xd}{\langle X^{\mr},\del\rangle}  %%% <X,¦Ä>
\newcommand{\Yr}{\langle Y^{\mr},\rho\rangle}  %%% <Y,¦Ñ>
\newcommand{\ush}{^\sharp}           %%% ^sharp
\newcommand{\ssh}{_\sharp}
\SelectTips{cm}{10}

\usepackage{enumitem}
\def\mod{\mathsf{mod}\hspace{.01in}}

\begin{document}
\baselineskip=15pt
\title{\Large{\bf Two new classes of $\bm{n}$-exangulated categories\footnotetext{\hspace{-1em}Jiangsheng Hu was supported by the NSF of China (Grant Nos. 11671069 and 11771212),  Qing Lan Project of Jiangsu Province and Jiangsu Government Scholarship for Overseas Studies (JS-2019-328). Panyue Zhou was supported by the National Natural Science Foundation of China (Grant Nos. 11901190 and 11671221), the Hunan Provincial Natural Science Foundation of China (Grant No. 2018JJ3205) and the Scientific Research Fund of Hunan Provincial Education Department (Grant No. 19B239).}}}
\medskip
\author{Jiangsheng Hu, Dongdong Zhang and Panyue Zhou}

\date{}

\maketitle
\def\blue{\color{blue}}
\def\red{\color{red}}

\newtheorem{theorem}{Theorem}[section]
\newtheorem{lemma}[theorem]{Lemma}
\newtheorem{corollary}[theorem]{Corollary}
\newtheorem{proposition}[theorem]{Proposition}
\newtheorem{conjecture}{Conjecture}
\theoremstyle{definition}
\newtheorem{definition}[theorem]{Definition}
\newtheorem{question}[theorem]{Question}
\newtheorem{remark}[theorem]{Remark}
\newtheorem{remark*}[]{Remark}
\newtheorem{example}[theorem]{Example}
\newtheorem{example*}[]{Example}
\newtheorem{condition}[theorem]{Condition}
\newtheorem{condition*}[]{Condition}
\newtheorem{construction}[theorem]{Construction}
\newtheorem{construction*}[]{Construction}

\newtheorem{assumption}[theorem]{Assumption}
\newtheorem{assumption*}[]{Assumption}

\baselineskip=17pt
\parindent=0.5cm

\begin{abstract}
\baselineskip=16pt
Herschend-Liu-Nakaoka introduced the notion of $n$-exangulated
categories. It is not only a higher
dimensional analogue of extriangulated categories defined by Nakaoka-Palu,
but also gives a simultaneous generalization of $n$-exact categories and $(n+2)$-angulated categories.
Let $\C$ be an $n$-exangulated category and $\X$ a full subcategory of $\C$.
If $\X$ satisfies $\X\subseteq\P\cap\I$, then { we give a necessary and sufficient condition for the ideal quotient $\C/\X$ to be an $n$-exangulated category},
where $\P$ (resp. $\I$) is the full subcategory of projective (resp. injective) objects in $\C$.
In addition, we define the notion of $n$-proper class in $\C$. If $\xi$ is an  $n$-proper class in $\C$, then we
prove that $\C$ admits a new $n$-exangulated structure.
These two ways give $n$-exangulated categories which are neither
$n$-exact nor $(n+2)$-angulated in general. \\[0.5cm]
\medskip
\end{abstract}

\pagestyle{myheadings}
\markboth{\rightline {\scriptsize J. Hu, D. Zhang, P. Zhou\hspace{2mm}}}
         {\leftline{\scriptsize  Two new classes of $n$-exangulated categories}}

\section{Introduction}

Higher-dimensional Auslander-Reiten theory was introduced by Iyama in \cite{I}, and it replaces short exact sequences as the basic building blocks for homological algebra, by the longer exact sequences. A typical setting is to consider $n$-cluster tilting subcategories of abelian categories (resp. exact categories), where $n$ is a positive integer. All short exact sequences in such a subcategory are split, but it has nice exact sequences with $n+2$ objects. This was recently formalized by Jasso \cite{J} in the theory of $n$-abelian categories (resp. $n$-exact categories). There exists also a derived version of the theory focusing on $n$-cluster tilting subcategories of triangulated categories as introduced by Geiss, Keller and Oppermann in the theory of $(n +2)$-angulated categories in \cite{GKO}. Setting $n=1$ recovers the notions of abelian, exact and triangulated categories. We refer to
\cite{BJT,BT,L,L2,LZ,ZW} for { more discussion} on this matter.

The class of extriangulated categories, recently introduced in \cite{NP}, not only contains exact categories and
extension-closed subcategories of triangulated categories as examples, but it is also
closed under taking some ideal quotients. This will help to
construct an extriangulated category which is neither exact nor triangulated,  see \cite[Proposition 3.30]{NP}, \cite[Example 4.14]{ZZ} and  \cite[Remark 3.3]{HZZ}. The data of such a category is a triplet $(\C,\E,\s)$, where $\C$ is an additive category, $\mathbb{E}: \C^{\rm op}\times \C \rightarrow {\rm Ab}$ is an additive bifunctor and $\mathfrak{s}$ assigns to each $\delta\in \mathbb{E}(C,A)$ a class of $3$-term sequences with end terms $A$ and $C$ such that certain axioms hold. Recently, Herschend-Liu-Nakaoka \cite{HLN} introduced an $n$-analogue of this notion called $n$-exangulated category. Such a category is a similar triplet $(\C,\E,\s)$, with the main distinction being that the $3$-term sequences mentioned above are replaced by $(n+2)$-term sequences. It should be noted that the case $n =1$ corresponds to extriangulated categories. As typical examples we have that $n$-exact and $(n+2)$-angulated categories are $n$-exangulated, see \cite[Propositions 4.5 and 4.34]{HLN}. However, there are some other examples of $n$-exangulated categories which are neither $n$-exact nor $(n+2)$-angulated, see \cite[Section 6]{HLN} and \cite[Remark 4.5]{LZ}.
The main purpose of this paper is to construct more classes of $n$-exangulated categories which are neither $n$-exact nor $(n+2)$-angulated.

We now outline the results of the paper. In Section \ref{section2}, we review some elementary definitions
and facts on $n$-exangulated categories.

In Section \ref{section3}, we assume that $(\C,\E,\s)$ is an $n$-exangulated category, and $\P$ (resp. $\I$) is the full subcategory of projective (resp. injective) objects in $(\C,\E,\s)$. If $\X$ is a full subcategory of $\C$ satisfying $\X\subseteq\P\cap\I$, then we give { a necessary and sufficient condition for the ideal quotient $\C/\X$ to be an $n$-exangulated category}, which allows us to construct a new class of $n$-exangulated categories which are neither $n$-exact nor $(n +2)$-angulated (see Theorem \ref{main} and Example \ref{example:3.4}).

In Section \ref{section4}, for a given $n$-exangulated category $(\C,\E,\s)$, we define a notion of an $n$-proper class of distinguished $n$-exangles, denoted by $\xi$.
If $(\C,\E,\s)$ is equipped with an $n$-proper class $\xi$ of distinguished $n$-exangles, then $(\C,\E,\s)$ admits a new $n$-exangulated structure (see Theorem \ref{thma}). It should be noted that the method here is different from the one used in \cite[Theorem 3.2]{HZZ} (see Remark \ref{remark:4.6}). This construction gives another new class of $n$-exangulated categories which are neither $n$-exact nor $(n +2)$-angulated (see Proposition \ref{proposition:4.8} and Example \ref{example:4.8}).

\section{Preliminaries}\label{section2}
Throughout this paper, { $\C$ is an additive category and $n$ is a positive integer. Suppose that $\C$ is equipped with an additive bifunctor $\E\colon\C\op\times\C\to\Ab$, where $\Ab$ is the category of abelian groups. Next we briefly recall some definitions and basic properties of $n$-exangulated categories from \cite{HLN}. We omit some
details here, but the reader can find them in \cite{HLN}.}

{ For any pair of objects $A,C\in\C$, an element $\del\in\E(C,A)$ is called an {\it $\E$-extension} or simply an {\it extension}. We also write such $\del$ as $\tensor[_A]{\del}{_C}$ when we indicate $A$ and $C$. The zero element $\tensor[_A]{0}{_C}=0\in\E(C,A)$ is called the {\it split $\E$-extension}. For any pair of $\E$-extensions $\tensor[_A]{\del}{_C}$ and $\tensor[_{A'}]{{{\delta}{'}}}{_{C'}}$, let $\delta\oplus \delta'\in\mathbb{E}(C\oplus C', A\oplus A')$ be the
element corresponding to $(\delta,0,0,{\delta}{'})$ through the natural isomorphism $\mathbb{E}(C\oplus C', A\oplus A')\simeq\mathbb{E}(C, A)\oplus\mathbb{E}(C, A')
\oplus\mathbb{E}(C', A)\oplus\mathbb{E}(C', A')$.

For any $a\in\C(A,A')$ and $c\in\C(C',C)$,  $\E(C,a)(\del)\in\E(C,A')\ \ \text{and}\ \ \E(c,A)(\del)\in\E(C',A)$ are simply denoted by $a_{\ast}\del$ and $c^{\ast}\del$, respectively.

Let $\tensor[_A]{\del}{_C}$ and $\tensor[_{A'}]{{{\delta}{'}}}{_{C'}}$ be any pair of $\E$-extensions. A {\it morphism} $(a,c)\colon\del\to{\delta}{'}$ of extensions is a pair of morphisms $a\in\C(A,A')$ and $c\in\C(C,C')$ in $\C$, satisfying the equality
$a_{\ast}\del=c^{\ast}{\delta}{'}$.}

\begin{definition}\cite[Definition 2.11]{HLN}
By { the} Yoneda lemma, any extension $\del\in\E(C,A)$ induces natural transformations
\[ \del\ssh\colon\C(-,C)\Rightarrow\E(-,A)\ \ \text{and}\ \ \del\ush\colon\C(A,-)\Rightarrow\E(C,-). \]
For any $X\in\C$, these $(\del\ssh)_X$ and $\del\ush_X$ are given as follows.
\begin{enumerate}
\item[\rm(1)] $(\del\ssh)_X\colon\C(X,C)\to\E(X,A)\ ;\ f\mapsto f\uas\del$.
\item[\rm (2)] $\del\ush_X\colon\C(A,X)\to\E(C,X)\ ;\ g\mapsto g\sas\delta$.
\end{enumerate}
\end{definition}

\begin{definition}\cite[Definition 2.7]{HLN}
Let $\bf{C}_{\C}$ be the category of complexes in $\C$. As its full subcategory, define $\CC$ to be the category of complexes in $\C$ whose components are zero in the degrees outside of $\{0,1,\ldots,n+1\}$. Namely, an object in $\CC$ is a complex $X^{\mr}=\{X^i,d_X^i\}$ of the form
\[ X^0\xrightarrow{d_X^0}X^1\xrightarrow{d_X^1}\cdots\xrightarrow{d_X^{n-1}}X^n\xrightarrow{d_X^n}X^{n+1}. \]
We write a morphism $f^{\mr}\co X^{\mr}\to Y^{\mr}$ { simply as} $f^{\mr}=(f^0,f^1,\ldots,f^{n+1})$, only indicating the terms of degrees $0,\ldots,n+1$.

{ We define the homotopy relation on the morphism sets in the usual way. Denote by $\mathbf{K}_{\C}^{n+2}$ the homotopy category, which is the quotient of $\CC$ by the ideal of null-homotopic morphisms.}
\end{definition}

\begin{definition}\cite[Definitions 2.9 and 2.13]{HLN}
 Let $\C,\E,n$ be as before. Define a category $\AE:=\AE^{n+2}_{(\C,\E)}$ as follows.
\begin{enumerate}
\item[\rm(1)]  A pair $\Xd$ is an object of the category $\AE$ with $X^{\mr}\in\CC$
and $\del\in\E(X^{n+1},X^0)$, called an {\it $\E$-attached
complex of length} $n+2$, if it satisfies
$$(d^0_X)_{\ast}\del=0~~\textrm{and}~~(d_X^n)^{\ast}\del=0.$$
We also denote it by
$$X^0\xrightarrow{d^0_X}X^1\xrightarrow{d^1_X}\cdots\xrightarrow{d^{n-2}_X}X^{n-1}
\xrightarrow{d^{n-1}_X}X^n\xrightarrow{d^n_X}X^{n+1}\overset{\delta}{\dashrightarrow}$$
\item[\rm (2)]  For such pairs $\Xd$ and $\langle Y^{\mr},\rho\rangle$, a morphism  $f^{\mr}\colon\Xd\to\langle Y^{\mr},\rho\rangle$ in $\AE$ is
defined to be a morphism in $\CC$ satisfying $(f^0)_{\ast}\del=(f^{n+1})^{\ast}\rho$.
{ We use the same composition and the identities as in $\CC$.}
\end{enumerate}

An {\it $n$-exangle} is a pair  $\Xd$ of $X^{\mr}\in\CC$ and $\delta\in\mathbb{E}(X^{n+1},X^0)$ which satisfies the following conditions.
\begin{enumerate}
\item[\rm (a)] The following sequence of functors $\C\op\to\Ab$ is exact.
$$
\C(-,X^0)\xLongrightarrow{\C(-,\ d_X^0)}\cdots\xLongrightarrow{\C(-,\ d_X^n)}\C(-,X^{n+1})\xLongrightarrow{~\del\ssh~}\E(-,X^0)
$$
\item[\rm (b)] The following sequence of functors $\C\to\Ab$ is exact.
$$
\C(X^{n+1},-)\xLongrightarrow{\C(d_X^n,\ -)}\cdots\xLongrightarrow{\C(d_X^0,\ -)}\C(X^0,-)\xLongrightarrow{~\del\ush~}\E(X^{n+1},-)
$$
\end{enumerate}
In particular any $n$-exangle is an object in $\AE$.
A {\it morphism of $n$-exangles} simply means a morphism in $\AE$. Thus $n$-exangles form a full subcategory of $\AE$.
\end{definition}

{ \begin{definition}\cite[Definition 2.17]{HLN}
Let $A,C\in{\C}$ be any pair of objects. The subcategory of $\CC$, denoted by $\mathbf{C}^{n+2}_{(\C;A,C)}$, or simply
by $\mathbf{C}^{n+2}_{(A,C)}$, is defined as follows.

\begin{enumerate}
\item[\rm(1)] An object $X^{\mr}\in{\CC}$ is in $\mathbf{C}^{n+2}_{(A,C)}$ if it satisfies $X^0=A$ and $X^{n+1}=C$. We also write it as $_{A}X^{\mr}_{C}$ when we emphasize $A$ and $C$.
\item[\rm(2)] For any $X^{\mr}, Y^{\mr}\in{\mathbf{C}^{n+2}_{(A,C)}}$, the morphism set is defined by
    $$\mathbf{C}^{n+2}_{(A,C)}(X^{\mr},Y^{\mr})=\{f^{\mr}\in{\CC(X^{\mr},Y^{\mr})}\ | \ f^{0}=1_{A}, \ f^{n+1}=1_{C}\}.$$
\end{enumerate}

We denote by $\mathbf{K}_{(A,C)}^{n+2}$ the quotient of $\mathbf{C}^{n+2}_{(A,C)}$
 by the same homotopy relation as $\mathbf{C}^{n+2}_{\C}$. The homotopy equivalence class of
$_{A}X^{\mr}_{C}$ is denoted by $[_{A}X^{\mr}_{C}]$, or simply by $[X^{\mr}]$.
\end{definition}}

{ \begin{remark}\cite[Remark 2.18]{HLN}
 Let $X^{\mr}, Y^{\mr}\in{\mathbf{C}^{n+2}_{(A,C)}}$ be any pair of objects. If a morphism $f^{\mr}\in\mathbf{C}^{n+2}_{(A,C)}(X^{\mr},Y^{\mr})$ gives a homotopy equivalence in $\mathbf{C}^{n+2}_{\C}$, then it is a homotopy equivalence in $\mathbf{C}^{n+2}_{(A,C)}$. However, the converse is not true in general. Thus there can be a difference between homotopy equivalences taken in ${\mathbf{C}^{n+2}_{(A,C)}}$ and $\mathbf{C}^{n+2}_{\C}$.
\end{remark}}

\begin{definition}\cite[Definition 2.22]{HLN}\label{def1}
Let $\s$ be a correspondence which associates a homotopy equivalence class $\s(\del)=[{}_AX^{\mr}_C]$ to each extension $\del={ \tensor[_A]{\delta}{_C}}$. Such { an} $\s$ is called a {\it realization} of $\E$ if it satisfies the following condition for any $\s(\del)=[X^{\mr}]$ and any $\s(\rho)=[Y^{\mr}]$.
\begin{itemize}
\item[{\rm (R0)}] For any morphism of extensions $(a,c)\co\del\to\rho$, there exists a morphism $f^{\mr}\in\CC(X^{\mr},Y^{\mr})$ of the form $f^{\mr}=(a,f^1,\ldots,f^n,c)$. Such { an} $f^{\mr}$ is called a {\it lift} of $(a,c)$.
\end{itemize}
In such a case, we { simply} say that \lq\lq$X^{\mr}$ realizes $\del$" whenever they satisfy $\s(\del)=[X^{\mr}]$.

Moreover, a realization $\s$ of $\E$ is said to be {\it exact} if it satisfies the following conditions.
\begin{itemize}
\item[{\rm (R1)}] For any $\s(\del)=[X^{\mr}]$, the pair $\Xd$ is an $n$-exangle.
\item[{\rm (R2)}] For any $A\in\C$, the zero element ${ \tensor[_A]{0}{_0}}=0\in\E(0,A)$ satisfies
\[ \s({ \tensor[_A]{0}{_0}})=[A\ov{1_A}{\lra}A\to0\to\cdots\to0\to0]. \]
Dually, $\s({ \tensor[_0]{0}{_A}})=[0\to0\to\cdots\to0\to A\ov{1_A}{\lra}A]$ holds for any $A\in\C$.
\end{itemize}
Note that the above condition {\rm (R1)} does not depend on representatives of the class $[X^{\mr}]$.
\end{definition}

\begin{definition}\cite[Definition 2.23]{HLN}
Let $\s$ be an exact realization of $\E$.
\begin{enumerate}
\item[\rm (1)] An $n$-exangle $\Xd$ is called an $\s$-{\it distinguished} $n$-exangle if it satisfies $\s(\del)=[X^{\mr}]$. We often simply say {\it distinguished $n$-exangle} when $\s$ is clear from the context.
\item[\rm (2)]  An object $X^{\mr}\in\CC$ is called an {\it $\s$-conflation} or simply a {\it conflation} if it realizes some extension $\del\in\E(X^{n+1},X^0)$.
\item[\rm (3)]  A morphism $f$ in $\C$ is called an {\it $\s$-inflation} or simply an {\it inflation} if it admits some conflation $X^{\mr}\in\CC$ satisfying $d_X^0=f$.
\item[\rm (4)]  A morphism $g$ in $\C$ is called an {\it $\s$-deflation} or simply a {\it deflation} if it admits some conflation $X^{\mr}\in\CC$ satisfying $d_X^n=g$.
\end{enumerate}
\end{definition}

\begin{definition}\cite[Definition 2.27]{HLN}
For a morphism $f^{\mr}\in\CC(X^{\mr},Y^{\mr})$ satisfying $f^0=1_A$ for some $A=X^0=Y^0$, its {\it mapping cone} $M_f^{\mr}\in\CC$ is defined to be the complex
\[ X^1\xrightarrow{d_{M_f}^0}X^2\oplus Y^1\xrightarrow{d_{M_f}^1}X^3\oplus Y^2\xrightarrow{d_{M_f}^2}\cdots\xrightarrow{d_{M_f}^{n-1}}X^{n+1}\oplus Y^n\xrightarrow{d_{M_f}^n}Y^{n+1} \]
where $d_{M_f}^0=\begin{bmatrix}-d_X^1\\ f^1\end{bmatrix},$
$d_{M_f}^i=\begin{bmatrix}-d_X^{i+1}&0\\ f^{i+1}&d_Y^i\end{bmatrix}\ (1\le i\le n-1),$
$d_{M_f}^n=\begin{bmatrix}f^{n+1}&d_Y^n\end{bmatrix}$.

{\it The mapping cocone} is defined dually, for morphisms $h^{\mr}$ in $\CC$ satisfying $h^{n+1}=1$.
\end{definition}

\begin{definition}\label{definition:2.10}\cite[Definition 2.32]{HLN}
An {\it $n$-exangulated category} is a triplet $(\C,\E,\s)$ of additive category $\C$, additive bifunctor $\E\co\C\op\times\C\to\Ab$, and its exact realization $\s$, satisfying the following conditions.
\begin{itemize}[leftmargin=4em]
\item[{\rm (EA1)}] Let $A\ov{f}{\lra}B\ov{g}{\lra}C$ be any sequence of morphisms in $\C$. If both $f$ and $g$ are inflations, then so is $g f$. Dually, if $f$ and $g$ are deflations then so is $g f$.

\item[{\rm (EA2)}] For $\rho\in\E(D,A)$ and $c\in\C(C,D)$, let ${}_A\langle X^{\mr},c\uas\rho\rangle_C$ and ${}_A\Yr_D$ be distinguished $n$-exangles. Then $(1_A,c)$ has a {\it good lift} $f^{\mr}$, in the sense that its mapping cone gives a distinguished $n$-exangle $\langle M^{\mr}_f,(d_X^0)\sas\rho\rangle$.
\item[{\rm (EA2$\op$)}] Dual of {\rm (EA2)}.
\end{itemize}
Note that the case $n=1$, a triplet $(\C,\E,\s)$ is a  $1$-exangulated category if and only if it is an extriangulated category, see \cite[Proposition 4.3]{HLN}.
\end{definition}

\begin{example}
From \cite[Proposition 4.34]{HLN} and \cite[Proposition 4.5]{HLN},  we know that $n$-exact categories and $(n+2)$-angulated categories are $n$-exangulated categories.
There are some other examples of $n$-exangulated categories
 which are neither $n$-exact nor $(n+2)$-angulated, see \cite[Section 6]{HLN} for more details.
\end{example}

\begin{definition}\label{def2}\cite[Definition 3.14]{ZW}
Let $(\C,\E,\s)$ be an $n$-exangulated category.
An object $P\in\C$ is called \emph{projective} if, for any distinguished $n$-exangle
$$A^0\xrightarrow{\alpha_0}A^1\xrightarrow{\alpha_1}A^2\xrightarrow{\alpha_2}\cdots\xrightarrow{\alpha_{n-2}}A^{n-1}
\xrightarrow{\alpha_{n-1}}A^n\xrightarrow{\alpha_n}A^{n+1}\overset{\delta}{\dashrightarrow}$$
and any morphism $c$ in $\C(P,A^{n+1})$, there exists a morphism $b\in\C(P,A^n)$ satisfying $\alpha_n b=c$.
We denote the full subcategory of projective objects in $\C$ by $\P$.
Dually, the full subcategory of injective objects in $\C$ is denoted by $\I$.
\end{definition}

\begin{remark} \cite[Remark 3.3]{LZ}\label{remark:2.12}
~\begin{itemize}
 \item[\rm (1)]  When $n=1$, they agree with
the usual definitions \cite[Definition 3.23]{NP}.

\item[\rm (2)] If $(\C,\E,\s)$ is an $n$-exact category, then they agree with \cite[Definition 3.11]{J}.

\item[\rm (3)] If $(\C,\E,\s)$ is an $(n+2)$-angulated category, then
$\P=\I$ consists of zero objects.
\end{itemize}
\end{remark}

\begin{lemma}\label{lem1}
Let $(\C,\E,\s)$ be an $n$-exangulated category. Then
$\E(\C,\P)=0$ and $\E(\I,\C)=0$.
\end{lemma}

\proof This follows from Lemma 3.4 and its dual in \cite{LZ}.  \qed

\section{Ideal quotients of $n$-exangulated categories}\label{section3}
Let $\X$ be a full subcategory of $\C$ { that is closed under isomorphisms
and finite direct sums.}
For two objects $A,B$ in { $\C$} denote by { $[\X](A,B)$} the subgroup of $\C(A,B)$ consisting of those morphisms which factor through an object in $\X$. Denote by $\C/\X$ the \emph{ideal quotient category} of $\C$ modulo $\X$: the objects are the same as the ones in $\C$, for two objects $A$ and $B$ the Hom space is given by the quotient group $\C(A,B)/{ [\X](A,B)}$.
Note that the ideal quotient category $\C/\X$ is an additive category.
We denote by $\overline{f}$ the image of $f\colon A\to B$ of $\C$ in $\C/\X$.
\medskip

Let $(\C, \mathbb{E}, \mathfrak{s})$ be an $n$-exangulated category and $\X$ a full subcategory of $\C$.
Assume that
$$A^0\xrightarrow{~\alpha_0~}A^1\xrightarrow{~\alpha_1~}A^2\xrightarrow{~\alpha_2~}\cdots\xrightarrow{~\alpha_{n-2}~}A^{n-1}
\xrightarrow{~\alpha_{n-1}~}A^n\xrightarrow{~\alpha_n~}A^{n+1}\overset{\delta}{\dashrightarrow}$$
is a distinguished $n$-exangle in $\C$. Denote $\overline{\C}:=\C/\X$.
{ This sequence
$$A^0\xrightarrow{~\overline{\alpha_0}~}A^1\xrightarrow{~\overline{\alpha_1}~}A^2\xrightarrow{~\overline{\alpha_2}~}\cdots
\xrightarrow{~\overline{\alpha_{n-2}}~}A^{n-1}
\xrightarrow{~\overline{\alpha_{n-1}}~}A^n\xrightarrow{~\overline{\alpha_n}~}A^{n+1}$$
is called \emph{weak kernel-cokernel sequence} if the following sequences
$$
\overline{\C}(-,A^0)\xLongrightarrow{\overline{\C}(-,\ \overline{\alpha_0})}\overline{\C}(-,A^1)\xLongrightarrow{\overline{\C}(-,\ \overline{\alpha_1})}\cdots\xLongrightarrow{\overline{\C}(-,\ \overline{\alpha_{n-1}})}\overline{\C}(-,A^n)\xLongrightarrow{\overline{\C}(-,\ \overline{\alpha_n})}\overline{\C}(-,A^{n+1})
$$
and
$$
\overline{\C}(A^{n+1},-)\xLongrightarrow{\overline{\C}(\overline{\alpha_n},\ -)}\overline{\C}(A^{n},-)\xLongrightarrow{\overline{\C}(\overline{\alpha_{n-1}},\ -)}\cdots\xLongrightarrow{\overline{\C}(\overline{\alpha_{1}},\ -)}\overline{\C}(A^1,-)\xLongrightarrow{\overline{\C}(\overline{\alpha_{0}},\ -)}\overline{\C}(A^0,-)
$$
are exact.}

The following construction gives $n$-exangulated categories which are { neither}
$n$-exact nor $(n+2)$-angulated in general.

\begin{theorem}\label{main}
Let $(\C, \mathbb{E}, \mathfrak{s})$ be an $n$-exangulated category and $\X$ a full subcategory of $\C$.
If $\X$ satisfies $\X\subseteq\P\cap\I$, then the ideal quotient $\C/\X$ is an $n$-exangulated category
{ if and only if any distinguished $n$-exangle in $\C$ induces a weak kernel-cokernel sequence in $\C/\X$.}
\end{theorem}

\proof Put $\overline{\C}=\C/\X$. By Lemma \ref{lem1}, we have
$\E(\C,\P)=0$ and $\E(\I,\C)=0$. Thus one can define the additive
bifunctor
$\overline{\E}\colon \overline{\C}^{\rm op}\times\overline{\C}\to\Ab$
given by
\begin{itemize}
\item $\overline{\E}(C,A)=\E(C,A)$ for any $A,C\in\C$,

\item $\overline{\E}(\overline{c},\overline{a})=\E(c,a)$
for any $a\in\C(A,A'),~c\in\C(C,C')$, where $\overline{a}$ and $\overline{c}$
denote the images of $a$ and $c$ in $\C/\X$, respectively.
\end{itemize}

For any $\overline{\E}$-extension $\delta\in\overline{\E}(C,A)={ \E}(C,A)$, define
$$\overline{\s}(\delta)=\overline{\s(\delta)}=[A\xrightarrow{~\overline{\alpha_0}~}
B^1\xrightarrow{~\overline{\alpha_1}~}
B^2\xrightarrow{~\overline{\alpha_2}~}\cdots\xrightarrow{~\overline{\alpha_{n-1}}~}B^n\xrightarrow{~\overline{\alpha_{n}}~}C]$$
using $\s(\delta)=[A\xrightarrow{~\alpha_0~}
B^1\xrightarrow{~\alpha_1~}
B^2\xrightarrow{~\alpha_2~}\cdots\xrightarrow{~\alpha_{n-1}~}B^n\xrightarrow{\alpha_{n}}C]$.
\smallskip

{ {\bf Necessity.~} It is trivial.}

{ {\bf Sufficiency.~}} Now we prove that $\overline{\s}$ is an exact realization of $\overline{\E}$.

\smallskip
Let $(\overline{a},\overline{c})\colon { \tensor[_A]{\delta}{_C}}\to {\tensor[_{A'}]{{{\delta}{'}}}{_{C'}}}$ be any morphism of $\overline{\E}$-extensions.
By definition, this is equivalent to that $(a,c)\colon {\tensor[_A]{\delta}{_C}}\to {\tensor[_{A'}]{{{\delta}{'}}}{_{C'}}}$ is a morphism of $\E$-extensions.
Put
$$\overline{\s}(\delta)=[B^{\mr}]=[A\xrightarrow{~\overline{\alpha_0}~}
B^1\xrightarrow{~\overline{\alpha_1}~}
B^2\xrightarrow{~\overline{\alpha_2}~}\cdots\xrightarrow{~\overline{\alpha_{n-1}}~}B^n\xrightarrow{\overline{\alpha_{n}}~}C],$$
$$\overline{\s}(\delta')=[{D}^{\mr}]=[A'\xrightarrow{~\overline{\beta_0}~}
D^1\xrightarrow{~\overline{\beta_1}~}
D^2\xrightarrow{~\overline{\beta_2}~}\cdots\xrightarrow{~\overline{\beta_{n-1}}~}D^n\xrightarrow{\overline{\beta_{n}}~}C'].$$
Since the condition in Definition \ref{def1} does not depend on the
representatives of the { {homotopy}} equivalence class, we may
assume
$$\s(\delta)=[A\xrightarrow{~\alpha_0~}
B^1\xrightarrow{~\alpha_1~}
B^2\xrightarrow{~\alpha_2~}\cdots\xrightarrow{~\alpha_{n-1}~}B^n\xrightarrow{\alpha_{n}~}C],$$
$$\s(\delta')=[A'\xrightarrow{~\beta_0~}
D^1\xrightarrow{~\beta_1~}
D^2\xrightarrow{~\beta_2~}\cdots\xrightarrow{~\beta_{n-1}~}D^n\xrightarrow{\beta_{n}~}C'].$$
Since $\s$ is a realization of $\E$,
there exists a morphism $f^{\mr}\in\CC(B^{\mr},{D}^{\mr})$ of the form $f^{\mr}=(a,f^1,\ldots,f^n,c)$
such that $f^{\mr}$ is a lift of $(a,c)$.
Thus $\overline{f^{\mr}}$ is a lift of $(\overline{a},\overline{c})$.
So (R0) is satisfied.

{ Now we prove that (R1) is satisfied.

For any
$$\overline{\s}(\delta)=[X^{\mr}]=[A\xrightarrow{~\overline{\alpha_0}~}
B^1\xrightarrow{~\overline{\alpha_1}~}
B^2\xrightarrow{~\overline{\alpha_2}~}\cdots\xrightarrow{~\overline{\alpha_{n-1}}~}B^n\xrightarrow{~\overline{\alpha_{n}}~}C],$$
through the above discussion, we can assume
$$\s(\delta)=[A\xrightarrow{~\alpha_0~}
B^1\xrightarrow{~\alpha_1~}
B^2\xrightarrow{~\alpha_2~}\cdots\xrightarrow{~\alpha_{n-1}~}B^n\xrightarrow{\alpha_{n}~}C].$$
It suffices to show that

\begin{equation}\label{equ1}\overline{\C}(-,B^{n})\xLongrightarrow{\overline{\C}(-,\ \overline{\alpha_{n}})}\overline{\C}(-,C)
\xLongrightarrow{~\overline{\del\ssh}~}\overline{\E}(-,A)
\end{equation}
and
\begin{equation}\label{equ2}\overline{\C}(B^1,-)
\xLongrightarrow{\overline{\C}(\overline{\alpha_{0}},\ -)}\overline{\C}(A,-)
\xLongrightarrow{~\overline{\del\ush}~}\overline{\E}(C,-)
\end{equation}
are exact.  We only prove that the sequence (\ref{equ1}) is exact, and the exactness of the sequence (\ref{equ2}) can be similarly proved.

For any $M\in\overline{\C}$, we define
\begin{eqnarray*}
\overline{(\del\ssh)}_{M}\colon \overline{\C}(M,C) & \longrightarrow &\overline{\E}(M,A)\\
 \overline{f}& \longmapsto & f^{\ast}\del
\end{eqnarray*}
which is well-defined. If $\overline{f}=\overline{g}$, then there are morphisms
$u\colon M\to X$ and $v\colon X\to C$ such that
$f-g=vu$ where $X\in\X$. Note that $X\in\X\subseteq\P\cap\I$, we have $v^{\ast}\del=0$.
Thus $(f-g)^{\ast}\del=(vu)^{\ast}\del=u^{\ast}v^{\ast}\del=0$ and then $f^{\ast}\del=g^{\ast}\del$.
\smallskip

$\bullet$ we show ${\rm Im}~\overline{\C}(M,\ \overline{\alpha_{n}})\subseteq{\rm Ker}~\overline{(\del\ssh)}_{M}$ for any $M\in\overline{\C}$.
\smallskip

For any morphism $\overline f\in {\rm Im}~\overline{\C}(M,\ \overline{\alpha_{n}})$, there exists a morphism $\overline\varphi\colon M\to B^n$
such that $\overline f=\overline{\alpha_{n}}\circ\overline\varphi=\overline{\alpha_{n}\varphi}$.
Hence
$\overline{(\del\ssh)}_{M}(\overline f)=(\alpha_{n}\varphi)^{\ast}\del=\varphi^{\ast}(\alpha_n)^{\ast}\del=0$ implies $\overline{f}\in{\rm Ker}~\overline{(\del\ssh)}_{M}$.
\smallskip

$\bullet$ we show ${\rm Ker}~\overline{(\del\ssh)}_{M}\subseteq{\rm Im}~\overline{\C}(M,\ \overline{\alpha_n})$ for any $M\in\overline{\C}$.
\smallskip

If $\overline f\in{\rm Ker}~\overline{(\del\ssh)}_{M}$,
then $f^{\ast}\del=0$. Thus there exists a morphism
$\phi\colon M\to B^n$ such that
$f=\alpha_n\phi$ and then $\overline f=\overline \alpha_n\circ\overline\phi$.
This shows $\overline f\in{\rm Im}~\overline{\C}(M,\ \overline{\alpha_n})$.
Thus the pair $\Xd$ is an $n$-exangle in $\overline{\C}$.}

(R2) are clearly satisfied.
This shows that $\overline{\s}$ is an exact realization of $\overline{\E}$.
\smallskip

Let us confirm conditions (EA1) and (EA2). The remaining condition (EA2$^{\rm op}$) can be shown dually.

{\rm (EA1)}~ Let $X\ov{\overline{f}}{\lra}Y\ov{\overline{g}}{\lra}Z$ be any sequence of morphisms in $\overline{\C}$.
Assume that $\overline{f}$ and $\overline{g}$ are inflations. Then there are two conflations $U^{\mr}\in\CC$ and $V^{\mr}\in\CC$ satisfying $\overline{d^U_0}=\overline{f}$ and $\overline{d^V_0}=\overline{g}$, respectively.
Thus we assume
$$\overline{\s}(\delta)=[X\xrightarrow{~\overline{f}~}
Y\xrightarrow{~\overline{d_U^1}~}
U^2\xrightarrow{~\overline{d_U^2}~}\cdots\xrightarrow{~\overline{d_U^{n-1}}~}U^n\xrightarrow{\overline{d_U^{n}}~}U^{n+1}],$$
$$\overline{\s}(\eta)=[Y\xrightarrow{~\overline{g}~}
Z\xrightarrow{~\overline{d_V^1}~}
V^2\xrightarrow{~\overline{d_V^2}~}\cdots\xrightarrow{~\overline{d_V^{n-1}}~}V^n\xrightarrow{\overline{d_V^{n}}~}V^{n+1}].$$
As in the proof of (R0), we may assume
$$\s(\delta)=[X\xrightarrow{~f~}
Y\xrightarrow{~d_U^1~}
U^2\xrightarrow{~d_U^2~}\cdots\xrightarrow{~d_U^{n-1}~}U_n\xrightarrow{d_U^{n}~}U^{n+1}],$$
$$\s(\eta)=[Y\xrightarrow{~g~}
Z\xrightarrow{~d_V^1~}
V^2\xrightarrow{~d_V^2~}\cdots\xrightarrow{~d_V^{n-1}~}V^n\xrightarrow{d_V^{n}~}V^{n+1}].$$
Then by (EA1) for $(\C,\E,\s)$, we know that $gf$ is an inflation.
Thus the image $\overline{g}\circ \overline{f}=\overline{gf}$ of $gf$ in $\overline{\C}$ is also an inflation.
Dually, we can show that if $\overline{f}$ and $\overline{g}$ are deflations, then so is $\overline{g}\circ \overline{f}$.
\medskip

{\rm (EA2)}~ For any $\delta\in\overline{\E}(D,A)$ and $\overline{c}\in\overline{\C}(C,D)$,
let ${}_A\langle X^{\mr},{ \overline{c}}\uas\delta\rangle_C$ and ${}_A\langle Y^{\mr}, \delta\rangle_D$ be distinguished $n$-exangles.
As in the proof of (R0), we may assume
$$\s(c^{\ast}\delta)=[A\xrightarrow{~d_X^0~}
X^1\xrightarrow{~d_X^1~}
X^2\xrightarrow{~d_X^2~}\cdots\xrightarrow{~d_X^{n-1}~}X^n\xrightarrow{d_X^{n}~}C],$$
$$\s(\delta)=[A\xrightarrow{~d_Y^0~}
Y^1\xrightarrow{~d_Y^1~}
Y^2\xrightarrow{~d_Y^2~}\cdots\xrightarrow{~d_Y^{n-1}~}Y^n\xrightarrow{d_Y^{n}~}D].$$
Then by (EA2) for $(\C,\E,\s)$,
 $(1_A,c)$ has a  good lift $f^{\mr}$, in the sense that its mapping cone gives a distinguished $n$-exangle $\langle M^{\mr}_f,(d_X^0)\sas\delta\rangle$. Thus the image of these conditions in
$\overline{\C}$ shows (EA2) for $(\overline{\C},\overline{\E},\overline{\s})$.              \qed

{ \begin{remark}
Let $(\C, \mathbb{E}, \mathfrak{s})$ be an extriangulated category and $\X$ a full subcategory of $\C$.
If $\X$ satisfies $\X\subseteq\P\cap\I$, then
any $\E$-triangle  in $\C$ induces a weak kernel-cokernel sequence in $\overline{\C}:=\C/\X$. Indeed,
assume that $A\xrightarrow{~\alpha_0~}B\xrightarrow{~\alpha_1~}C\overset{\delta}{\dashrightarrow}$
is an $\E$-triangle in $\C$.
We claim that
 \begin{equation}\label{equ3}
\overline{\C}(-,A)\xLongrightarrow{\overline{\C}(-,\ \overline{\alpha_0})}\overline{\C}(-,B)\xLongrightarrow{\overline{\C}(-,\ \overline{\alpha_1})}\overline{\C}(-,C)
 \end{equation}
and
 \begin{equation}\label{equ4}
\overline{\C}(C,-)\xLongrightarrow{\overline{\C}(\overline{\alpha_1},\ -)}\overline{\C}(B,-)\xLongrightarrow{\overline{\C}(\overline{\alpha_{0}},\ -)}\overline{\C}(A,-)
 \end{equation}
are exact.

We only prove that the sequence (\ref{equ3}) is exact, and the exactness of the sequence (\ref{equ4}) can be similarly proved.

Since $\alpha_1\alpha_0=0$, we have $\overline\alpha_1\circ\overline\alpha_0=0$, which implies
${\rm Im}~\overline{\C}(-,\ \overline{\alpha_0})\subseteq {\rm Ker}~\overline{\C}(-,\ \overline{\alpha_1})$.

Conversely, for any $M\in\overline{\C}$ and $\overline{a}\in{\rm Ker}~\overline{\C}(M,\ \overline{\alpha_1})$, we have
$\overline\alpha_1\circ\overline a=0$. Then there are morphisms $b\colon M\to X$
and $c\colon X\to C$ such that $\alpha_1a=cb$ where $X\in\X\subseteq\P\cap\I$.
Since $X$ is a projective object,
there exists a morphism $d\colon X\to B$ such that $c=\alpha_1d$.
It follows that $\alpha_1(a-db)=\alpha_1a-\alpha_1db=0$.
So there exists a morphism
$e\colon M\to A$ such that $a-db=\alpha_0e$ and then
$\overline a=\overline\alpha_0\circ \overline e\in{\rm Im}~\overline{\C}(M,\ \overline{\alpha_0})$.
This shows that $A\xrightarrow{~\overline{\alpha_0}~}B\xrightarrow{~\overline{\alpha_1}~}C$
is a weak kernel-cokernel sequence.
By Theorem \ref{main}, we have that the ideal quotient $\C/\X$ is an extriangulated category.

This result is just the Proposition 3.30 in \cite{NP}.
Hence our result recovers and extends a result of Nakaoka and Palu \cite[Proposition 3.30]{NP}.
\end{remark}}

%As a direct consequence of Theorem \ref{main}, we have the following.

%\begin{corollary}
%Let $\C$ be a Frobenius $n$-exangulated category. Then the ideal quotient $\C/\I$ is an $(n+2)$-angulated category.
%\end{corollary}
%
%\proof This follows from Theorem \ref{main} and \cite[Proposition 4.8]{HLN}.  \qed
\medskip

{ The following example shows that the ideal quotient $\C/\X$ in Theorem \ref{main} is not an $n$-exangulated category in general.}

 \begin{example}\label{example:3.4}
Let $\Lambda$ be the path algebra of the
quiver
$$1\xrightarrow{~\alpha~}2\xrightarrow{~\beta~}3\xrightarrow{~\gamma~}4$$
with relation $\alpha\beta\gamma=0$.
Then $\mod\Lambda$ has a unique $2$-cluster tilting subcategory $\C$ consisting of all direct
sums of projective modules and injective modules. By \cite[Theorem 3.16]{J}, we know that $\C$ is a $2$-abelian category which can be viewed as
a $2$-exangulated category.
The Auslander-Reiten quiver of $\mod\Lambda$ is the following
$$
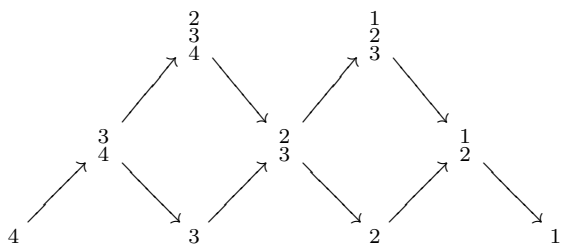
\xymatrix@C=0.7cm@R0.7cm{
&&&{\begin{smallmatrix}
2\\  3\\4
\end{smallmatrix}} \ar[dr]
&&{\begin{smallmatrix}
1\\2\\3
\end{smallmatrix}} \ar[dr]
&& \\
&&{\begin{smallmatrix}
3\\4
\end{smallmatrix}}\ar[dr] \ar[ur]
&& {\begin{smallmatrix}
2\\3
\end{smallmatrix}} \ar[dr] \ar[ur]
&& {\begin{smallmatrix}
1\\2
\end{smallmatrix}} \ar[dr] \\
&{\begin{smallmatrix}
4
\end{smallmatrix}}\ar[ur]
&&{\begin{smallmatrix}
3
\end{smallmatrix}} \ar[ur]
&&  {\begin{smallmatrix}
2
\end{smallmatrix}} \ar[ur]
&& {\begin{smallmatrix}
1
\end{smallmatrix}}}$$
That is $\C:=\add\{\begin{smallmatrix}
4
\end{smallmatrix}, \begin{smallmatrix}
3\\ 4
\end{smallmatrix},\begin{smallmatrix}
2\\ 3\\ 4
\end{smallmatrix},\begin{smallmatrix}
1\\ 2 \\3
\end{smallmatrix},\begin{smallmatrix}
1\\ 2
\end{smallmatrix},\begin{smallmatrix}
1
\end{smallmatrix}\}$. Take
$$\X:=\add\{\begin{smallmatrix}
2\\ 3\\ 4
\end{smallmatrix}\}\subseteq \add\{ \begin{smallmatrix}
2\\ 3 \\4
\end{smallmatrix}, \begin{smallmatrix}
1\\ 2 \\ 3
\end{smallmatrix}\}=\P\cap\I.$$
Note that $\begin{smallmatrix}
4
\end{smallmatrix}\to \begin{smallmatrix}
2\\ 3\\ 4
\end{smallmatrix}\to \begin{smallmatrix}
1\\ 2
\end{smallmatrix}\to \begin{smallmatrix}
1
\end{smallmatrix}$
is a $2$-exact sequence in $\C$, { but it induces the sequence
$$\begin{smallmatrix}
4
\end{smallmatrix}\to \begin{smallmatrix}
0
\end{smallmatrix}\to \begin{smallmatrix}
1\\ 2
\end{smallmatrix}\to \begin{smallmatrix}
1
\end{smallmatrix}$$
which is not a weak kernel-cokernel sequence in $\C/\X$.
Therefore $\C/\X$ equipped with $\overline{\mathbb{E}}$ and
$\overline{\s}$ is not a $2$-exangulated category by Theorem \ref{main}.}
\end{example}

Now { we} give an example to explain our main result in this section.

{ \begin{example}\label{example:3.4}
Let $\Lambda$ be the path algebra of the
quiver
$$1\xrightarrow{~\alpha~}2\xrightarrow{~\beta~}3$$
with relation $\alpha\beta=0$.
Then $\mod\Lambda$ has a unique $2$-cluster tilting subcategory $\C$ consisting of all direct
sums of projective  modules and injective modules. By \cite[Theorem 3.16]{J}, we know that $\C$ is a $2$-abelian category which can be viewed as
a $2$-exangulated category.
The Auslander-Reiten quiver of  $\mod\Lambda$ is the following
$$\xymatrix@C=0.7cm@R0.7cm{
&{\begin{smallmatrix}
2\\ 3
\end{smallmatrix}}
\ar[dr]&&{\begin{smallmatrix}
1\\2
\end{smallmatrix}}\ar[dr]\\
{\begin{smallmatrix}
3
\end{smallmatrix}}\ar[ur]&&{\begin{smallmatrix}
2
\end{smallmatrix}}\ar[ur]&&{\begin{smallmatrix}
1
\end{smallmatrix}}}$$
That is $\C:=\add\{\begin{smallmatrix}
3
\end{smallmatrix}, \begin{smallmatrix}
2\\ 3
\end{smallmatrix},\begin{smallmatrix}
1\\2
\end{smallmatrix},\begin{smallmatrix}
1
\end{smallmatrix}\}$.
Take
$$\X:=\add\{\begin{smallmatrix}
2\\ 3
\end{smallmatrix},\begin{smallmatrix}
1\\2
\end{smallmatrix}\}=\P\cap\I.$$
We know that there exists a unique non-trivial $2$-exact sequence
$\begin{smallmatrix}
3
\end{smallmatrix}\to \begin{smallmatrix}
2\\ 3
\end{smallmatrix}\to \begin{smallmatrix}
1\\2
\end{smallmatrix}\to \begin{smallmatrix}
1
\end{smallmatrix}$
in $\C$. It follows that
$\begin{smallmatrix}
3
\end{smallmatrix}\to \begin{smallmatrix}
0
\end{smallmatrix}\to \begin{smallmatrix}
0
\end{smallmatrix}\to \begin{smallmatrix}
1
\end{smallmatrix}$
is a weak kernel-cokernel sequence in $\C/\X$.
By Theorem \ref{main}, we obtain that $$\C/\X=\add\{\begin{smallmatrix}
1
\end{smallmatrix}, \begin{smallmatrix}
3
\end{smallmatrix}\}$$
 equipped with $\overline{\mathbb{E}}$ and
$\overline{\s}$ is a $2$-exangulated category, but it is neither
a $2$-exact category nor a $4$-angulated category.
In fact, $\begin{smallmatrix}
3
\end{smallmatrix}\to \begin{smallmatrix}
2\\ 3
\end{smallmatrix}\to \begin{smallmatrix}
1\\ 2
\end{smallmatrix}\to \begin{smallmatrix}
1
\end{smallmatrix}$
is a $2$-exact sequence in $\C$, but it induces the sequence
$$\begin{smallmatrix}
3
\end{smallmatrix}\to \begin{smallmatrix}
0
\end{smallmatrix}\to \begin{smallmatrix}
0
\end{smallmatrix}\to \begin{smallmatrix}
1
\end{smallmatrix}$$
 which is not a 2-exact sequence in $\C/\X$. So $(\C/\X,\overline{\mathbb{E}},\overline{\s})$ is not 2-exact.
Since $\C/\X$ has non-zero projective objects,
 it is not a $4$-angulated category.
\end{example}}

\section{$n$-proper classes in $n$-exangulated categories}\label{section4}
In this section, we introduce a notion of $n$-proper class in an $n$-exangulated category,
and give a new class of an $n$-exangulated category.
Unless otherwise specified, we assume that $(\C,\E,\s)$ is an $n$-exangulated category.

A morphism $f^{\mr}\colon\Xd\to\langle Y^{\mr},\rho\rangle$ of distinguished $n$-exangles is called a \emph{weak isomorphism} if
{$f^{0}$ and $f^{n+1}$} are isomorphisms. { It should be noted that the notion of weak isomorphisms defined here is different from the one defined by Geiss, Keller and Oppermann in $(n +2)$-angulated categories or by Jasso  in $n$-exact categories (see \cite{GKO} and \cite{J})}.

\begin{proposition}\label{proposition:4.1} If $f^{\mr}\colon\Xd\to\langle Y^{\mr},\rho\rangle$ is a weak isomorphism of distinguished $n$-exangles, then $f^{\mr}:X^{\mr}\to Y^{\mr}$ is a homotopy equivalence in $\CC$.
\end{proposition}

{\begin{proof} If  $f^{\mr}\colon\Xd\to\langle Y^{\mr},\rho\rangle$ is a weak isomorphism of distinguished $n$-exangles, then there exist $g^0:Y^0\to X^0$ and $g^{n+1}:Y^{n+1}\to X^{n+1}$ such that $g^0f^0=1_{X^0}, f^0g^0=1_{Y^0}, g^{n+1}f^{n+1}=1_{X^{n+1}}$ and  $f^{n+1}g^{n+1}=1_{Y^{n+1}}$. Note that $(f^0)_*\delta=(f^{n+1})^*\rho$. Then $(g^0)_*\rho=(g^0)_{*}(f^{n+1}g^{n+1})^{\ast}\rho$ $=
(g^0)_{*}(g^{n+1})^{\ast}(f^{n+1})^{\ast}\rho=(g^0)_{*}(g^{n+1})^{\ast}(f^{0})_{\ast}\delta
=(g^0)_{*}(f^{0})_{\ast}(g^{n+1})^{\ast}\delta=
(g^0f^{0})_{\ast}(g^{n+1})^{\ast}\delta=(g^{n+1})^{\ast}\delta$. Hence there is a morphism $g^{\mr}:\langle Y^{\mr}, \rho\rangle \to \langle X^{\mr},\delta\rangle$ of distinguished $n$-exangles where $g^{\mr}$ is a lift of $(g^0,g^{n+1})$.  Therefore $g^{\mr}f^{\mr}\colon\Xd\to\langle X^{\mr},\delta\rangle$ is a morphism of distinguished $n$-exangles with $g^0f^0=1_{X^0}$ and $g^{n+1}f^{n+1}=1_{X^{n+1}}$, it
follows from \cite[Proposition 2.21]{HLN} that $g^{\mr} f^{\mr}$ is a homotopy equivalence. Similarly, one can prove that $f^{\mr} g^{\mr}$ is a homotopy equivalence. Consequently, there exist $h^{\mr},\lambda^{\mr}\in{\CC(Y^{\mr},X^{\mr})}$ such that $ h^{\mr}f^{\mr}\sim 1_{X^{\mr}}$ and  $f^{\mr}\lambda^{\mr}\sim 1_{Y^{\mr}}$. This implies that $ f^{\mr}$ is a homotopy equivalence in $\CC$.
\end{proof}}

{ Let $\xi$ be a class of distinguished $n$-exangles. A distinguished $n$-exangle $\Xd$ is called a \emph{$\xi$-distinguished $n$-exangle} if it belongs to $\xi$. A morphism $f$ is called a \emph{$\xi$-inflation} (resp. \emph{$\xi$-deflation}) if there is a $\xi$-distinguished $n$-exangle $\Xd$ such that $d_X^0=f$ (resp. $d_X^n=f$).}

A class $\xi$ of distinguished $n$-exangles  is called {\it closed under  base change} if for any $\xi$-distinguished $n$-exangle $_{A}\Xd_{C}$ and any morphism $c:C'\to C$, then any distinguished $n$-exangle $_{A}{\langle Y^{\mr},c^{\ast}\del\rangle}_{C'}$ belongs to $\xi$. Dually, $\xi$ is called {\it closed under cobase change} if for any $\xi$-distinguished $n$-exangle $_{A}\Xd_{C}$ and any morphism $a:A\to A'$, then any distinguished $n$-exangle $_{A'}{\langle Y^{\mr},a_{\ast}\del\rangle}_{C}$ belongs to $\xi$.

A class $\xi$ of distinguished $n$-exangles is called {\it saturated} if for any commutative diagram in $\C$
$$\xymatrix{
     A\ar[d]_{c} \ar[r]^{a} & B\ar[d]^{b} \\
  C \ar[r]^{d} & D,\\}
  $$
  where $a,b$ and $d$ are deflations. If $a$ and $b$ are $\xi$-deflations, then so is $d$.

A distinguished $n$-exangle $_{A}\Xd_{C}$ is called {\it split} if $\delta=0$. It follows from \cite[Claim 2.15]{HLN} that it is split if and only if $d^{0}_{X}$ is a section or $d^{n}_{X}$ is a retraction. The full subcategory  consisting of the split distinguished $n$-exangles will be denoted by $\Delta_0$.

\begin{definition}\label{definition:4.2}
 A class $\xi$ of distinguished $n$-exangles is called an {\it $n$-proper class} if the following conditions hold:
\begin{enumerate}
\item[\rm (1)] $\Delta_0\subseteq \xi$ and $\xi$ is closed under weak isomorphisms and finite coproducts.

\item[\rm (2)] $\xi$ is closed under base change and cobase change.

\item[\rm (3)] $\xi$ is saturated.
\end{enumerate}
\end{definition}

 { If $(\C, \mathcal{E})$ is an $n$-exact category and $\xi$ is a class of $n$-exact sequences, then $\xi$ is called an {\it $n$-proper class} in $(\C, \mathcal{E})$  if the class of distinguished $n$-exangles induced by $\xi$ is an $n$-proper class in the $n$-exangulated category induced by $(\C, \mathcal{E})$. Similarly, we have the definition of an {\it $n$-proper class} for any $(n+2)$-angulated category.

}

The following lemma is  very important, which will be used later.

\begin{lemma}\label{lemma:4.3}Let $\xi$ be a class of distinguished $n$-exangles in $(\C,\E,\s)$ satisfying the conditions {\rm (1)} and {\rm (2)} in Definition {\rm \ref{definition:4.2}}. { Then the following are equivalent.
\begin{enumerate}
\item[\rm (1)] $\xi$ is saturated.

\item[\rm (2)] For any commutative diagram in $\C$
$$\xymatrix{
     A\ar[d]_{c} \ar[r]^{a} & B\ar[d]^{b} \\
  C \ar[r]^{d} & D,\\}
  $$
  where $a,b$ and $c$ are inflations. If $a$ and $b$ are $\xi$-inflations, then so is $c$.
\end{enumerate}}
\end{lemma}
\begin{proof} { We only show $(1)\Rightarrow (2)$, and the proof of $(2)\Rightarrow (1)$ is similar.}

$(1)\Rightarrow (2)$. Let $_{A}\Xd_{M}$ be a $\xi$-distinguished $n$-exangle with $d_{X}^{0}=a$, $X^0=A$, and $X^1=B$, and $_{B}{\langle Y^{\mr},\theta\rangle}_{N}$ a $\xi$-distinguished $n$-exangle with $d_{Y}^{0}=b$, $Y^0=B$ and $Y^1=D$. Since $d_{X}^{0}$ and $d_{Y}^{0}$ are inflations, their composition $d_{Y}^{0}d_{X}^{0}$ becomes an inflation by (EA1) in Definition \ref{definition:2.10}. Thus there is some distinguished $n$-exangle $_{A}{\langle Z^{\mr},\tau\rangle}_{L}$ which satisfies $Z^1=D$ and $d_{Z}^{0}=d_{Y}^{0}d_{X}^{0}$ as follows

\[ A\xrightarrow{d_{Z}^{0}}D\xrightarrow{d_{Z}^{1}}Z^{2}\xrightarrow{d_{Z}^{2}}\cdots\xrightarrow{d_{Z}^{n-1}}Z^{n}\xrightarrow{d_Z^{n}}L\overset{\tau}{\dashrightarrow}. \]
Applying \cite[Proposition 3.5]{HLN}, we have the following commutative diagram
$$\xymatrix{A\ar[r]^{d_{X}^{0}}\ar@{=}[d]&B\ar[r]^{d_{X}^{1}}\ar[d]^{d_{Y}^{0}}&X^{2}\ar@{-->}[d]^{f^{2}}\ar[r]&\cdots \ar[r]& X^{n}\ar@{-->}[d]^{f^{n}}\ar[r]^{d_{X}^{n}}& M \ar@{-->}[d]^{f^{n+1}} \ar@{-->}[r]^{\delta}&\\
   A\ar[r]^{d_{Z}^{0}}&D\ar[r]^{d_{Z}^{1}}&Z^{2}\ar[r]&\cdots \ar[r]& Z^{n}\ar[r]^{d_{Z}^{n}}& L  \ar@{-->}[r]^{\tau}&.}$$
Thus we have a morphism $f^{\mr}\colon\Xd\to\langle Z^{\mr},\tau\rangle$ which satisfies $f^0=1_{A}$, $f^{1}=d_{Y}^{0}$ and makes $_{B}\langle M^{\mr}_f,(d_X^0)\sas\tau\rangle_{L}$ a distinguished $n$-exangle. Applying \cite[Proposition 3.5]{HLN} again, we have the following commutative diagram
$$\xymatrix{B\ar[r]^{d_{M_{f}}^{0}}\ar@{=}[d]& X^{2}\oplus D \ar[r]\ar[d]^{\tiny\begin{bmatrix}0&1\end{bmatrix}}&X^{3}\oplus Z^{2}\ar@{-->}[d]^{g^{2}}\ar[r]&\cdots \ar[r]& M\oplus Z^{n}\ar@{-->}[d]^{g^{n}}\ar[r]^{d_{M_{f}}^{n}}& L \ar@{-->}[d]^{g^{n+1}} \ar@{-->}[r]^{(d_X^0)\sas\tau}&\\
   B\ar[r]^{d_{Y}^{0}}&D\ar[r]^{d_{Y}^{1}}&Y^{2}\ar[r]&\cdots \ar[r]& Y^{n}\ar[r]^{d_{Y}^{n}}& N  \ar@{-->}[r]^{\theta}&,}$$
   where $d_{M_f}^0=\begin{bmatrix}-d_X^1\\ d_Y^{0}\end{bmatrix},$
$d_{M_f}^n=\begin{bmatrix}f_{n+1}&d_Z^n\end{bmatrix}$. Hence we have a morphism $g^{\mr}\colon{\langle M^{\mr}_f,(d_X^0)\sas\tau\rangle}\to\langle Y^{\mr},\theta\rangle$ which satisfies $g^0=1_{B}$, $g^{1}=\begin{bmatrix}0&1\end{bmatrix}$ and makes $_{X^{2}\oplus D}\langle M^{\mr}_g, {\tiny\begin{bmatrix}-d_X^1\\ d_Y^{0}\end{bmatrix}}\sas \theta\rangle_{N}$ a distinguished $n$-exangle. In particular, $(d_{X}^{0})\sas\tau=(g^{n+1})^{\ast}\theta$. Since ${\langle Y^{\mr},\theta\rangle}$ is a $\xi$-distinguished $n$-exangle and $\xi$ is closed under base change, $\langle M^{\mr}_f,(d_X^0)\sas\tau\rangle$ is a $\xi$-distinguished $n$-exangle.

Note that $0\to0\to\cdots\to0\to Z^{n}\ov{1_{Z^{n}}}{\lra}Z^{n}\overset{0}{\dashrightarrow}$ is a split distinguished $n$-exangle. Then it is a distinguished $n$-exangle in $\xi$. Thus $X^0\xrightarrow{{ d_{X}^{0}}}X^{1}\xrightarrow{d_{X}^{1}}X^{2}\xrightarrow{d_{X}^{2}}\cdots\xrightarrow{d_{X}^{n-1}}X^{n}\oplus Z^{n}\xrightarrow{\tiny \begin{bmatrix}{ d_X^{n}}&0\\ 0&1_{Z^{n}}\end{bmatrix}}M\oplus Z^{n}\overset{\delta}{\dashrightarrow}$ is a distinguished $n$-exangle in $\xi$ since $\xi$ is closed under finite coproducts by hypothesis. Consider the following commutative diagram
$$\xymatrix@C=1.6cm@R0.8cm{
     X^{n}\oplus Z^{n}\ar[d]_{{\tiny\begin{bmatrix}f^{n}&1 \end{bmatrix}}} \ar[r]^{\tiny \begin{bmatrix}d_X^{n}&0\\ 0&1_{Z^{n}}\end{bmatrix}} & M\oplus Z^{n}\ar[d]^{\tiny\begin{bmatrix}f^{n+1}&d_{Z}^{n} \end{bmatrix}} \\
 Z^{n} \ar[r]^{d_{Z}^{n}} & L,\\}
  $$
  where  $d_{Z}^{n}$, $\begin{bmatrix}d_X^{n}&0\\ 0&1_{Z^{n}}\end{bmatrix}$, $\begin{bmatrix}f^{n+1}&d_{Z}^{n} \end{bmatrix}$ are deflations. Note that $\begin{bmatrix}f^{n+1}&d_{Z}^{n} \end{bmatrix}$ and $\begin{bmatrix}d_X^{n}&0\\ 0&1_{Z^{n}}\end{bmatrix}$ are $\xi$-deflations. Then $d_{Z}^{n}$ is a $\xi$-deflation because $\xi$ is saturated. Assume that $_{E}{\langle W^{\mr},\rho\rangle}_{L}$ is a $\xi$-distinguished $n$-exangle with $d_{W}^{n}=d_{Z}^{n}$ and $W^n=Z^{n}$. By the dual of \cite[Proposition 3.5]{HLN}, we have the following commutative diagram
  $$\xymatrix{E\ar[r]^{d_{W}^{0}}\ar@{-->}[d]^{h^{0}}&W^{1}\ar[r]^{d_{W}^{1}}\ar@{-->}[d]^{h^{1}}&W^{2}\ar@{-->}[d]^{h^{2}}\ar[r]&\cdots \ar[r]& Z^{n}\ar@{=}[d]\ar[r]^{d_{Z}^{n}}& L \ar@{=}[d] \ar@{-->}[r]^{\rho}&\\
   A\ar[r]^{d_{Z}^{0}}&D\ar[r]^{d_{Z}^{1}}&Z^{2}\ar[r]&\cdots \ar[r]& Z^{n}\ar[r]^{d_{Z}^{n}}& L  \ar@{-->}[r]^{\tau}&.}$$
   Thus we have a morphism $h^{\mr}\colon{\langle W^{\mr},\rho\rangle}\to\langle Z^{\mr},\tau\rangle$ which satisfies $h^{n}=1_{Z^{n}}$, $h^{n+1}=1_{L}$ and $(h^{0})_{\ast}\rho=\tau$, and hence $\langle Z^{\mr},\tau\rangle$ is a $\xi$-distinguished $n$-exangle because $\xi$ is closed under cobase change. Since $c:A\to C$ is an inflation, there is a distinguished $n$-exangle $_{A}{\langle V^{\mr},\gamma\rangle}_{K}$  with $d_{V}^{0}=c$ and $V^1=C$. By \cite[Proposition 3.5]{HLN}, we have the following commutative diagram $$\xymatrix{A\ar[r]^{c}\ar@{=}[d]&C\ar[r]^{d_{V}^{1}}\ar[d]^{d}&V^{2}\ar@{-->}[d]^{s^{2}}\ar[r]&\cdots \ar[r]& V^{n}\ar@{-->}[d]^{s^{n}}\ar[r]^{d_{V}^{n}}& K \ar@{-->}[d]^{s^{n+1}} \ar@{-->}[r]^{\gamma}&\\
   A\ar[r]^{d_{Z}^{0}}&D\ar[r]^{d_{Z}^{1}}&Z^{2}\ar[r]&\cdots \ar[r]& Z^{n}\ar[r]^{d_{Z}^{n}}& L  \ar@{-->}[r]^{\tau}&.}$$
   Thus we have a morphism $s^{\mr}\colon{\langle V^{\mr},\gamma\rangle}\to\langle Z^{\mr},\tau\rangle$ which satisfies $s^{0}=1_{A}$ and $s^{1}=d$ such that $\gamma=(s^{n+1})^{\ast}\tau$, and hence $\gamma$ is a $\xi$-distinguished $n$-exangle because $\xi$ is closed under base change. So $c$ is a $\xi$-inflation, as desired.
   \end{proof}

By the proof of Lemma \ref{lemma:4.3}, we have the following corollary.
\begin{corollary}\label{corollary 4.4}  Let $\xi$ be an $n$-proper class in $(\C,\E,\s)$.
 Then the class of $\xi$-inflations {\rm(}resp. $\xi$-deflations{\rm)} is closed under compositions.
\end{corollary}

The following is our main result of this section.

\begin{theorem}\label{thma} Let $\xi$ be a class of distinguished $n$-exangles in $(\C,\E,\s)$ which is closed under weak isomorphism.
Set $\mathbb{E}_\xi:=\mathbb{E}|_\xi$, that is, $$\mathbb{E}_\xi(C, A)=\{\delta\in\mathbb{E}(C, A)~|~\delta~ \textrm{is realized as a distinguished n-exangle} \ {_{A}\Xd_{C}} \ \textrm{in} \ {\xi}\}$$ for any $A, C\in{\C}$, and $\mathfrak{s}_\xi:=\mathfrak{s}|_{\mathbb{E}_\xi}$. Then $\xi$ is an  $n$-proper class if and only if $(\C, \mathbb{E}_\xi, \mathfrak{s}_\xi)$ is an $n$-exangulated category.
\end{theorem}
\begin{proof} ``$\Rightarrow$'' If $\delta\in\mathbb{E}_\xi(C,A)$ with $\mathfrak{s}_\xi(\delta)=[Y^{\mr}]$, then there is a distinguished $n$-exangle $_A\langle X^{\mr},\delta\rangle_C$ in $\xi$ by the definition of $\mathbb{E}_\xi(C,A)$. Note that $_A\langle Y^{\mr},\delta\rangle_C$ is a distinguished $n$-exangle. Then there is a morphism $f^{\mr}: {_A\langle X^{\mr},\delta\rangle_C}\to {_A\langle Y^{\mr},\delta\rangle_C}$ with $f^0=1_A$ and $f^{n+1}=1_C$. It is clear that $f^{\mr}$ is a weak isomorphism, hence $_A\langle Y^{\mr},\delta\rangle_C$ is a distinguished $n$-exangle in $\xi$.

Next we show that $\mathbb{E}_\xi$ is an additive sub-bifunctor of $\mathbb{E}$. For any $\delta, \rho\in \mathbb{E}_\xi(C,A)$, we have $\delta+\rho=\begin{bmatrix}1&1\end{bmatrix}_*\begin{bmatrix}1\\1\end{bmatrix}^*(\delta\oplus \rho)\in \mathbb{E}_\xi(C,A)$
because $\xi$ is closed under base change, cobase change and finite coproducts. It is easy to see that  $\mathbb{E}_\xi$ is an additive sub-bifunctor of $\mathbb{E}$. It follows from \cite[Claim 3.8]{HLN} and Corollary \ref{corollary 4.4} that $(\C, \mathbb{E}_\xi, \mathfrak{s}_\xi)$ is an $n$-exangulated category.

``$\Leftarrow$'' Note that $(\C, \mathbb{E}_\xi, \mathfrak{s}_\xi)$ is an $n$-exangulated category by hypothesis. It is easy to check that $\xi$ satisfies the conditions (1) and (2) in Definition \ref{definition:4.2}. Next we claim that $\xi$ is  saturated. Consider the following commutative diagram in $\C$
$$\xymatrix{
     A\ar[d]_{c} \ar[r]^{a} & B\ar[d]^{b} \\
  C \ar[r]^{d} & D,\\}
  $$
  where $a,b$ and $c$ are inflations. Assume that $a$ and $b$ are $\xi$-inflations. Then $f=ba$ is a $\xi$-inflation. Thus there is a $\xi$-distinguished $n$-exangle $_{A}{\langle Z^{\mr},\theta\rangle}_{M}$ with $d_{Z}^{0}=f$. Note that $c$ is an inflation. Then there is a distinguished $n$-exangle $_{A}{\langle X^{\mr},\rho\rangle}_{N}$ with $d_{X}^{0}=c$. By \cite[Proposition 3.5]{HLN}, we have the following commutative diagram $$\xymatrix{A\ar[r]^{c}\ar@{=}[d]&C\ar[r]^{d_{X}^{1}}\ar[d]^{d}&X^{2}\ar@{-->}[d]^{g^{2}}\ar[r]&\cdots \ar[r]& X^{n}\ar@{-->}[d]^{g^{n}}\ar[r]^{d_{X}^{n}}& N \ar@{-->}[d]^{g^{n+1}} \ar@{-->}[r]^{\rho}&\\
   A\ar[r]^{f}&D\ar[r]^{d_{Z}^{1}}&Z^{2}\ar[r]&\cdots \ar[r]& Z^{n}\ar[r]^{d_{Z}^{n}}& M \ar@{-->}[r]^{\theta}&.}$$
   Thus we have a morphism $g^{\mr}\colon{\langle X^{\mr},\rho\rangle}\to\langle Z^{\mr},\theta\rangle$ which satisfies $g^{0}=1_{A}$ and $g^{1}=d$ and $(g^{n+1})^{\ast}\theta=\rho$. Since $\xi$ is closed under base change by the proof above, $\langle X^{\mr},\rho\rangle$ is a $\xi$-distinguished $n$-exangle. So $\xi$ is saturated by Lemma \ref{lemma:4.3}, as desired.
\end{proof}

\begin{remark}\label{remark:4.6}
(1) Assume that $\xi$ is a class of distinguished $n$-exangles in $(\C,\E,\s)$ which is closed under weak isomorphisms. By Theorem \ref{thma} and \cite[Proposition 3.14]{HLN}, one can check that $\xi$ is an $n$-proper class if and only if $\mathbb{E}_\xi$ is a closed additive subfunctor of $\E$ defined by Herschend-Liu-Nakaoka in \cite[Definition 3.9 and Lemma 3.13]{HLN}.

(2) In Theorem \ref{thma}, when $n=1$, it is just the Theorem 3.2 in \cite{HZZ}. Note that, in \cite{HZZ}, one of the key arguments in the proof is that any extriangulated category has shifted octahedrons, while in our general context we do not have this fact and therefore must avoid this kind of arguments. So, the idea of proving Theorem \ref{thma} { is} different from the one in \cite{HZZ}.
\end{remark}

If we choose $(\C,\E,\s)$ to be an $n$-exact category or an $(n+2)$-angulated category, then we have the following corollary which is a consequence of  Theorem \ref{thma}.

\begin{corollary}\label{corollary:4.6} The following are true for any  $n$-exangulated category $(\C,\E,\s)$:

\begin{enumerate}
\item[\rm (1)] If $(\C,\E,\s)$ is an $n$-exact category and $\xi$ is a class of $n$-exact sequences which is closed under weak isomorphisms, then $\xi$ is an $n$-proper class if and only if $(\C, \mathbb{E}_\xi, \mathfrak{s}_\xi)$ is an $n$-exact category.

\item[\rm (2)] If $(\C,\E,\s)$ is an $(n+2)$-angulated category and $\xi$ is a class of $(n+2)$-angles which is closed under weak isomorphisms, then $\xi$ is an $n$-proper class if and only if $(\C, \mathbb{E}_\xi, \mathfrak{s}_\xi)$ is an $n$-exangulated category.
\end{enumerate}
\end{corollary}
\begin{proof} { (1) The ``if" part holds by Theorem \ref{thma}, so we will prove the ``only if" part.

Assume that $\xi$ is an $n$-proper class. It follows from Theorem \ref{thma} that $(\C, \mathbb{E}_\xi, \mathfrak{s}_\xi)$ is an $n$-exangulated category. Note that any distinguished $n$-exangle in $(\C, \mathbb{E}_\xi, \mathfrak{s}_\xi)$ is induced by $n$-exact sequences in $\xi$, hence any  $\xi$-inflation is monomorphic and any $\xi$-deflation is epimorphic in $\C$. It follows from \cite[Corollary 4.13, Proposition 4.23 and Proposition 4.37]{HLN} and Proposition \ref{proposition:4.1} that $(\C, \mathbb{E}_\xi, \mathfrak{s}_\xi)$ is an $n$-exact category.

(2) It is clear by Theorem \ref{thma}.}
\end{proof}

Let $\C$ be an additive category and $\H$  a subcategory of $\C$. Recall that a morphism $f:A\to B$ in $\C$ is called a \emph{left $\H$-approximation} of $A$
if $B\in{\H}$ and
$${\C}(f, H):{\C}(B,H)\to {\C}(A,H)$$ is an epimorphism for any $H\in{\mathscr{H}}$.
Moreover, if $(\C,\Sigma,\Theta)$ is an $(n+2)$-angulated category, then we say that a subcategory $\H$ of $\C$ is called
\emph{strongly covariantly finite} if for any object $B\in\C$, there exists an $(n+2)$-angle
$$B\xrightarrow{~f~}H^1\xrightarrow{}H^2\xrightarrow{}\cdots\xrightarrow{}H^{n-1}\xrightarrow{}H^{n}\xrightarrow{~~}C\xrightarrow{~~}\Sigma B$$
where $f$ is a left $\H$-approximation of $B$ and $H^1, H^2,\cdots, H^n\in\H$.

The following construction gives $n$-exangulated categories which are neither $n$-exact nor $(n+2)$-angulated.

\begin{proposition}\label{proposition:4.8} Let $(\C,\Sigma,\Theta)$ be an $(n+2)$-angulated category and  $\mathscr{H}$ a  full subcategory of $\C$. Denote by $\xi$ the class of $(n+2)$-angles $$X^0\xrightarrow{d_{X}^{0}}X^{1}\xrightarrow{d_{X}^{1}}X^{2}\xrightarrow{d_{X}^{2}}
\cdots\xrightarrow{d_{X}^{n-1}}X^{n}\xrightarrow{d_X^{n}}X^{n+1}\xrightarrow{~\delta~}\Sigma{X^{0}}$$
such that ${\C}(d_{X}^{0},H):{\C}(X^1,H)\to {\C}(X^{0},H)$ is an epimorphism for any $H\in{\mathscr{H}}$. Then the following statements hold.

\begin{enumerate}
\item[\rm (1)] $\xi$ is an $n$-proper class in $(\C,\Sigma,\Theta)$ which induces an $n$-exangulated category $(\C, \mathbb{E}_\xi, \mathfrak{s}_\xi)$.

\item[\rm (2)] If $(\C, \mathbb{E}_\xi, \mathfrak{s}_\xi)$ is an $n$-exact category, then $\xi$ is the class of split $(n+2)$-angles.

\item[\rm (3)] If $(\C, \mathbb{E}_\xi, \mathfrak{s}_\xi)$ is an $(n+2)$-angulated category, then $\Theta=\xi$.

\item[\rm (4)] Assume that { $\mathscr{H}$ is a strongly covariantly finite subcategory of $\C$} which is closed under direct summands. If $\{0\}\neq \mathscr{H}\subsetneqq \C$, then $(\C, \mathbb{E}_\xi, \mathfrak{s}_\xi)$ is neither $n$-exact nor $(n+2)$-angulated.
\end{enumerate}
\end{proposition}
\begin{proof} (1) First we show that the class of distinguished $n$-exangles induced by $\xi$ is closed under weak isomorphisms. Assume that $ X^{\mr}$ is an  $(n+2)$-angle in $\xi$ and $f^{\mr}: \langle X^{\mr},\delta\rangle\to \langle Y^{\mr},\rho\rangle$ is a weak isomorphism of distinguished $n$-exangles in $(\C, \mathbb{E}, \mathfrak{s})$ where $\langle X^{\mr},\delta\rangle$ and $ \langle Y^{\mr},\rho\rangle$ are induced by $(n+2)$-angles $X^{\mr}$ and $Y^{\mr}$, respectively. By the proof of Proposition \ref{proposition:4.1}, there is a morphism $g^{\mr}: \langle Y^{\mr}, \rho\rangle\to \langle X^{\mr},\delta\rangle$ of distinguished $n$-exangles such that $1_{Y^0}=f^0g^0$. For any morphism $\beta: Y^0\to H$ with $H\in \mathscr{H}$, there is a morphism $\gamma: X^1\to H$ such that $\beta f^0=\C(d_X^0, H)(\gamma)=\gamma d_X^0$ because $\C(d_{X}^{0}, H):\C(X^1, H)\to \C(X^0, H)$ is an epimorphism. Consider the exact sequence $$\C(Y^1, H)\xrightarrow{\C(d_Y^0, H)}\C(Y^0, H)\xrightarrow{\C(\Sigma^{-1}\rho, H)}\C(\Sigma^{-1}Y^{n+1}, H).$$ We obtain $\C(\Sigma^{-1}\rho, H)(\beta)=\beta\Sigma^{-1}\rho=\beta 1_{Y_0}\Sigma^{-1}\rho=\beta f^0g^0\Sigma^{-1}\rho=\gamma d_{X}^{0}g^0\Sigma^{-1}\rho=\gamma g_1d_{Y}^{0}\Sigma^{-1}\rho=0$. Hence $\C(d_{Y}^{0}, H):\C(Y^1, H)\to \C(Y^0, H)$ is an epimorphism, and $Y^{\mr}$ is an $(n+2)$-angle in $\xi$.

By Corollary \ref{corollary:4.6}(2), it suffices to show that $\xi$ is closed under base change and cobase change, and $\xi$ is saturated. Let $X^0\xrightarrow{d_{X}^{0}}X^{1}\xrightarrow{d_{X}^{1}}X^{2}\xrightarrow{d_{X}^{2}}\cdots\xrightarrow{d_{X}^{n-1}}X^{n}\xrightarrow{d_X^{n}}X^{n+1}\xrightarrow{\delta}\Sigma{X^{0}}$ be an $(n+2)$-angle in $\xi$. For any morphism $q:M\to X^{n+1}$, we have the following commutative diagram
$$\xymatrix{X^{0}\ar[r]^{d_{Y}^{0}}\ar@{=}[d]&Y^{1}\ar[r]^{d_{Y}^{1}}\ar@{-->}[d]^{f^{1}}&Y^{2}\ar@{-->}[d]^{f^{2}}\ar[r]&\cdots \ar[r]& Y^{n}\ar@{-->}[d]^{f^{n}}\ar[r]^{d_{Y}^{n}}& M \ar[d]^{q} \ar[r]^{{ \delta q}}&\Sigma X^{0}\ar@{=}[d]\\
   X^{0}\ar[r]^{d_{X}^{0}}&X^{1}\ar[r]^{d_{X}^{1}}&X^{2}\ar[r]&\cdots \ar[r]& X^{n}\ar[r]^{d_{X}^{n}}& X^{n+1}  \ar[r]^{\delta}&\Sigma X^{0}.}$$
   Let $H$ be any object in $\mathscr{H}$. Then we have the following commutative diagram
   $$\xymatrix@C=2cm@R1cm{
     {\C}(X^{1},H)\ar[d]^{\C(f^{1},H)} \ar[r]^{{\C}(d_{X}^{0},H)} & {\C}(X^{0},H)\ar@{=}[d] \\
  {\C}(Y^{1},H) \ar[r]^{{\C}(d_{Y}^{0},H)} & {\C}(X^{0},H).\\}
  $$
Note that ${\C}(d_{X}^{0},H):{\C}(X^{1},H)\to {\C}(X^{0},H)$ is an epimorphism. It follows that ${\C}(d_{Y}^{0},H):{\C}(Y^{1},H)\to {\C}(X^{0},H)$ is an epimorphism, which implies that $\xi$ is closed under base change. To prove that $\xi$ is closed under cobase change, we assume that $l:X^{0}\to N$ is a morphism in $\C$. Then we have the following commutative diagram
$$\xymatrix{X^{0}\ar[r]^{d_{X}^{0}}\ar[d]^{l}&X^{1}\ar[r]^{d_{X}^{1}}\ar@{-->}[d]^{g^{1}}&X^{2}\ar@{-->}[d]^{g^{2}}\ar[r]&\cdots \ar[r]& X^{n}\ar@{-->}[d]^{g^{n}}\ar[r]^{d_{X}^{n}}& X^{n+1} \ar@{=}[d] \ar[r]^{\delta }&\Sigma X^{0}\ar[d]^{\Sigma l}\\
  N\ar[r]^{d_{Z}^{0}}&Z^{1}\ar[r]^{d_{Z}^{1}}&Z^{2}\ar[r]&\cdots \ar[r]& Z^{n}\ar[r]^{d_{Z}^{n}}& X^{n+1}  \ar[r]^{(\Sigma l) \delta}&\Sigma N.}$$
   Let $H$ be any object in $\mathscr{H}$. Then we have the following commutative diagram with exact rows
   $$\xymatrix@C=3.4cm@R1cm{
     {\C}(Z^{1},H)\ar[d]^{{\C}(g^{1},H)} \ar[r]^{{\C}(d_{Z}^{0},H)} &  {\C}(N,H)\ar[d]^{{\C}(l,H)} \ar[r]^{{\C}((-1)^{n}(l\Sigma^{-1}\delta),H)\quad} & {\C}(\Sigma^{-1}X^{n+1},H)\ar@{=}[d] \\
 {\C}(X^{1},{ H}) \ar[r]^{{\C}(d_{X}^{0},H)} &  {\C}(X^{0},H) \ar[r]^{\C((-1)^{n}\Sigma^{-1}\delta,H)\quad} & {\C}(\Sigma^{-1}X^{n+1},H).\\}
  $$
 Since ${\C}(d_{X}^{0},H):{\C}(X^{1},H)\to {\C}(X^{0},H)$ is an epimorphism, it follows that ${\C((-1)^{n}\Sigma^{-1}\delta,H)}=0.$
 Thus ${{\C}((-1)^{n}(l\Sigma^{-1}\delta),H)}=0$, and hence ${\C}(d_{Z}^{0},H):{\C}(Z^{1},H)\to {\C}(N,H)$ is an epimorphism, as desired.

  Finally, for any commutative diagram in $\C$
$$\xymatrix{
     A\ar[d]_{c} \ar[r]^{a} & B\ar[d]^{b} \\
  C \ar[r]^{d} & D,\\}
  $$
  where $a,b$ and $c$ are inflations in $(\C,\Sigma,\Theta)$,  we have the following commutative diagram
  $$\xymatrix@C=2cm@R1cm{
    {\C}(D,H)\ar[d]_{{\C}(b,H)} \ar[r]^{{\C}(d,H)} & {\C}(c,H)\ar[d]^{{\C}(c,H)} \\
  {\C}(B,H) \ar[r]^{{\C}(a,H)} & {\C}(A,H),\\}
  $$
  where $H$ is any object in $\mathscr{H}$. Let $a$ and $b$ be $\xi$-inflations. Then ${{\C}(a,H)}$ and ${{\C}(b,H)}$ are epimorphisms. Thus ${{\C}(c,H)}{{\C}(d,H)}$ is an epimorphism, and hence ${{\C}(c,H)}$ is an epimorphism. So $\xi$ is saturated by Lemma \ref{lemma:4.3}.

  (2) Assume that $(\C, \mathbb{E}_\xi, \mathfrak{s}_\xi)$ is an $n$-exact category. Then for any { $(n+2)$-angle} $$X^0\xrightarrow{d_{X}^{0}}X^{1}\xrightarrow{d_{X}^{1}}X^{2}\xrightarrow{d_{X}^{2}}
  \cdots\xrightarrow{d_{X}^{n-1}}X^{n}\xrightarrow{d_X^{n}}X^{n+1}\xrightarrow{~\delta~}\Sigma{X^{0}}$$ in $\xi,$  one has $X^0\xrightarrow{d_{X}^{0}}X^{1}\xrightarrow{d_{X}^{1}}X^{2}\xrightarrow{d_{X}^{2}}\cdots\xrightarrow{d_{X}^{n-1}}X^{n}\xrightarrow{d_X^{n}}X^{n+1}$ is an $n$-exact sequence in $\C$. Thus $d_{X}^{0}$ is a split monomorphism by \cite[Lemma 2.4]{L}, and hence $\delta=0$ by \cite[Lemma 2.3]{L}. So $\xi$ is the class of split $(n+2)$-angles.

  (3) The result holds by \cite[Proposition 2.5]{GKO}.

  (4)  Assume that $\mathscr{H}$ is a strongly covariantly finite subcategory of $(\C,\Sigma,\Theta)$ which is closed under direct summands. It is easy to check that the class of  injective objects in $(\C, \mathbb{E}_\xi, \mathfrak{s}_\xi)$ equals to $\mathscr{H}$. If $(\C, \mathbb{E}_\xi, \mathfrak{s}_\xi)$ is an $n$-exact category, then $\xi$ is the class of split $(n+2)$-angles by (2). Thus $\mathscr{H}=\C$, a contradiction. On the other hand, if $(\C, \mathbb{E}_\xi, \mathfrak{s}_\xi)$ is an $(n+2)$-angulated category, then $\mathscr{H}$ consists of zero objects by { Remark \ref{remark:2.12}(3)}. This yields a contradiction.
\end{proof}

Now we give a concrete example to explain our main result in this section.

\begin{example}\label{example:4.8}
This example comes from \cite{L2}.
Let $\T=D^b(kQ)/\tau^{-1}[1]$ be the cluster category of type $A_3$,
where { $k$ is an algebraically closed field,} $Q$ is the quiver $1\xrightarrow{~\alpha~}2\xrightarrow{~\beta~}3$,
$D^b(kQ)$ is the bounded derived category of finitely generated modules over $kQ$, $\tau$ is the Auslander-Reiten translation and $[1]$ is the shift functor
of $D^b(kQ)$. Then $\T$ is a $2$-Calabi-Yau triangulated category. Its shift functor is also denoted by $[1]$.

We describe the
Auslander-Reiten quiver of $\T$ in the following:
$$
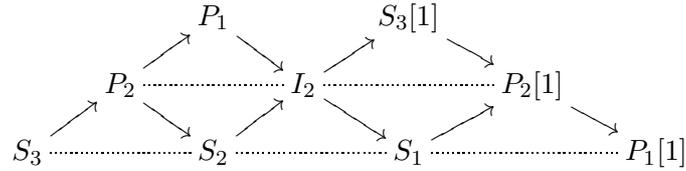
\xymatrix@C=0.6cm@R0.3cm{
&&P_1\ar[dr]
&&S_3[1]\ar[dr]
&&\\
&P_2 \ar@{.}[rr] \ar[dr] \ar[ur]
&&I_2 \ar@{.}[rr] \ar[dr] \ar[ur]
&&P_2[1]\ar[dr]\\
S_3\ar[ur]\ar@{.}[rr]&&S_2\ar[ur]\ar@{.}[rr]
&&S_1\ar[ur]\ar@{.}[rr]&&P_1[1]
}
$$
It is straightforward to verify that $\C:=\add(S_3\oplus P_1\oplus S_1)$ is a $2$-cluster tilting subcategory of $\T$. Moreover, $\C[2]=\C$.  By \cite[Theorem 1]{GKO}, we know that $(\C,[2])$ is a $4$-angulated category.
Let $\mathscr{H}=\add(S_3\oplus S_1)$.
Then the $4$-angle
$$P_1\xrightarrow{~~}S_1\xrightarrow{~~}S_3\xrightarrow{~~}P_1\xrightarrow{~~}P_1[2]$$
shows that $\mathscr{H}$ is a strongly covariantly finite subcategory of $\C$. Note that $\{0\}\neq \mathscr{H}\subsetneqq \C$. So $(\C, \mathbb{E}_\xi, \mathfrak{s}_\xi)$ is neither $n$-exact nor $(n+2)$-angulated by Proposition \ref{proposition:4.8}.
\end{example}

\section*{\bf Acknowledgments}

The authors would like to thank the referee for reading the paper carefully and for many suggestions on mathematics and English expressions.

\textbf{Jiangsheng Hu}\\
School of Mathematics and Physics, Jiangsu University of Technology,
 Changzhou, Jiangsu 213001, P. R. China.\\
E-mail: \textsf{jiangshenghu@jsut.edu.cn}\\[1mm]
\textbf{Dongdong Zhang}\\
Department of Mathematics, Zhejiang Normal University,
 Jinhua, Zhejiang 321004, P. R. China.\\
E-mail: \textsf{zdd@zjnu.cn}\\[1mm]
\textbf{Panyue Zhou}\\
College of Mathematics, Hunan Institute of Science and Technology, 414006, Yueyang, Hunan, P. R. China.\\
E-mail: \textsf{panyuezhou@163.com}

\end{document}